\documentclass[a4paper, oneside]{amsart}

\usepackage{amsthm, amsmath, amssymb, amstext, amsfonts}
\usepackage[top=2cm, bottom=2cm]{geometry}
\usepackage{hyperref}
\usepackage{enumitem}
\usepackage[svgnames]{xcolor}
\usepackage{accents}%%%
\definecolor{smoked}{RGB}{216, 212, 204}
\definecolor{mauve}{RGB}{200, 55, 171}
\definecolor{apricot}{RGB}{250, 144, 4}
\definecolor{sky}{RGB}{66, 169, 244}
\definecolor{plum}{RGB}{76, 0, 102}

\definecolor{lightmauve}{RGB}{232, 173, 220}
\definecolor{lightapricot}{RGB}{253, 211, 155}
\definecolor{lightsky}{RGB}{178, 221, 251}
\definecolor{lightplum}{RGB}{184, 153, 192}

\definecolor{darksmoked}{RGB}{198, 194, 176}

\hypersetup{
   colorlinks,
    linkcolor={mauve},
    citecolor={plum},
    urlcolor={plum}
}
\usepackage{graphicx}
\usepackage{tikz,tqft}
\usepackage{tkz-euclide}
\usetikzlibrary{matrix,arrows,bending,cd,decorations.markings,calc,patterns,angles,quotes,backgrounds,math}
\tikzset{>=latex}
\usepackage{caption}
\theoremstyle{plain}
\newtheorem{thm}{Theorem}[section]
\newtheorem{lem}[thm]{Lemma}
\newtheorem{prop}[thm]{Proposition}
\newtheorem{fact}[thm]{Fact}
\newtheorem{thmx}{Theorem}

\newtheorem{corx}[thmx]{Corollary}

\theoremstyle{remark}
\newtheorem{rmk}[thm]{Remark}

\theoremstyle{definition}
\newtheorem{defn}[thm]{Definition}
\newtheorem{ex}[thm]{Example}

\numberwithin{equation}{section}

\DeclareMathOperator{\id}{id}

\DeclareMathOperator{\R}{\mathbb{R}}
\DeclareMathOperator{\C}{\mathbb{C}}
\DeclareMathOperator{\Z}{\mathbb{Z}}

\DeclareMathOperator{\CP}{\mathbb{CP}}
\DeclareMathOperator{\HH}{\mathbb{H}}

\newcommand{\PSO}{\operatorname{PSO}}
\newcommand{\PSL}{\operatorname{PSL}}
\newcommand{\PU}{\operatorname{PSU}}
\newcommand{\psl}{\PSL_2 \R}

\newcommand{\surface}{\Sigma}
\newcommand{\Rep}{\operatorname{Rep}}
\newcommand{\Hom}{\operatorname{Hom}}

\title[Totally elliptic surface group representations]{Totally elliptic surface group representations}

\author[Arnaud Maret]{Arnaud Maret} 


\address{Université de Strasbourg, IRMA, 7 rue Descartes, 67000 Strasbourg,
France}
\email{maret.arnaud@unistra.fr}

\begin{document}

\begin{abstract}
A surface group representation into a Lie group is totally elliptic if every simple closed curve on the surface is mapped to an elliptic element. In this note, we characterize all totally elliptic surface group representations into $\psl$ and $\PSL_2\C$ by showing that they are either representations into a compact subgroup or Deroin--Tholozan representations. 
\end{abstract}

\maketitle

\section{Introduction}
\subsection{Main results}
Surface group representations are homomorphisms from the fundamental group of a surface $\surface$ into a Lie group $G$. The topological quotient of the space of all homomorphisms $\pi_1\surface \to G$ by the conjugation action of $G$ will be called the \emph{character variety} of the pair $(\surface, G)$ (even though the quotient may not be a Hausdorff topological space) and denoted by
\[
\Rep (\surface, G)= \Hom(\pi_1\surface, G)/G.
\] 
A character variety is not necessarily a connected space; dynamical behaviour and qualitative properties of representations may vary considerably between components. At one end of the spectrum are components of discrete and faithful representations (so-called \emph{higher Teichmüller spaces}).  This article focuses on the opposite end of the spectrum, where connected components are mostly made of non-discrete representations.

We study a specific prototype of surface group representations---known as totally elliptic representations---which are antipodal to discrete and faithful ones.
\begin{defn}\label{def:totally-elliptic}
A surface group representation $\pi_1\surface\to G$ is \emph{totally elliptic} if every element of $\pi_1\surface$ representing a simple closed curve is mapped to an elliptic element of $G$, i.e.~an element of a compact subgroup of $G$.
\end{defn}
The surface $\surface$ will always be assumed to be oriented and connected, and may be punctured. A \emph{simple closed curve} on $\surface$ is the free isotopy class of a closed loop on $\surface$ that does not self-intersect and is homotopically non-trivial. A simple closed curve can be homotopic to a puncture, in which case we call it a \emph{peripheral curve}. By definition of a totally elliptic representation, the images of all peripheral curves are elliptic. We want to avoid the situation where peripheral curves are mapped to the identity element in $G$ and we say that a representation is \emph{reduced} when this is not the case (Definition~\ref{def:reduced-representation}).

The first examples of totally elliptic representations, though not the most striking, are those whose image lies in a compact subgroup of $G$. Interestingly, by requiring only \emph{simple} closed curves to be mapped to elliptic elements (as opposed to all curves), there exist totally elliptic representations with dense image into non-compact Lie groups and whose deformations remain totally elliptic. This turns out to be possible whenever $G$ is a Hermitian Lie group (Section~\ref{sec:perspectives}).

This paper focuses on the case where $G$ is $\psl$ or $\PSL_2\C$. The first totally elliptic representations with (Zariski) dense image in $\psl$ were discovered by Benedetto--Goldman~\cite{benedetto-goldman} for 4-punctured spheres, and later generalized by Deroin--Tholozan~\cite{deroin-tholozan} for spheres with arbitrarily many punctures. Those representations were originally named \emph{supra-maximal} in~\cite{deroin-tholozan}, but we shall call them \emph{DT representations} (Definition~\ref{def:DT-representation}), to avoid any confusion with maximal representations from higher Teichmüller theory. Deroin--Tholozan proved that deforming conjugacy classes of DT representations, while preserving conjugacy classes of peripheral monodromies, leads to compact components inside the corresponding \emph{relative character varieties} (Section~\ref{sec:genus-zero}) which are made exclusively of totally elliptic representations (Theorem~\ref{thm:DT}). The question was raised whether DT representations are the only totally elliptic representations into $\psl$, besides representations into compact subgroups which we call \emph{orthogonal representations}. We bring a positive answer to that question.

\begin{thmx}\label{thm-intro}
Consider a surface $\surface$ of genus $g \geq 0$ with $n \geq 0$ punctures, and let $\rho\colon \pi_1\surface \to \psl$ be a reduced totally elliptic representation. If $\rho$ is not orthogonal, then $\Sigma$ must be a sphere with at least three punctures $(g = 0$, $n \geq 3)$, and $\rho$ is a DT representation.
\end{thmx}

We prove Theorem~\ref{thm-intro} in Section~\ref{sec:proof-of-main-theorem}. The case of surfaces of genus at least $1$ is straightforward. In that case, if $\surface$ is different than the closed torus, then it admits simple closed curves that can be represented by the commutator of two fundamental group elements, each representing a simple closed curve (Fact~\ref{fact:commutator-scc-is-scc}). However, the commutator of two elliptic elements in $\psl$ is elliptic if and only if it is trivial (Lemma~\ref{lem:commutator-hyperbolic}). These two observations will obstruct the existence of non-orthogonal totally elliptic representations when $g\geq 1$. The case of genus-$0$ surfaces in Theorem~\ref{thm-intro} is more subtle. When $\surface$ is a punctured sphere, we will prove that totally elliptic and non-orthogonal representations are DT representations by inducting on the number $n$ of punctures on $\surface$. The base case $n=4$ is a consequence of a combination of results by Deroin--Tholozan (Proposition~\ref{prop:totally-elliptic-bounded}) and Cantat--Loray (Theorem~\ref{thm:cantat-loray}). For the induction step, we will require the notion of \emph{triangle chains} developed in~\cite{action-angle} and the associated characterization of DT representations (Proposition~\ref{prop:DT-if-triangle-coherently-oriented}). 

\subsection{Applications and generalizations}
An immediate consequence of Theorem~\ref{thm-intro} is that DT representations can be characterized inside their relative character variety by the property of being totally elliptic.

\begin{corx}\label{cor-intro-characterization-DT}
If $\surface$ is a sphere with $n\geq 3$ punctures and $\rho\colon\pi_1\surface\to \psl$ maps every peripheral curve on $\surface$ to a regular (i.e.~non-trivial) elliptic element of $\psl$, then
\[
\rho\text{ is a DT representation}\quad\Longleftrightarrow\quad \rho\text{ is totally elliptic and non-orthogonal.}
\]
\end{corx}

The direct implication in Corollary~\ref{cor-intro-characterization-DT} was established in~\cite[Lemma~3.2]{deroin-tholozan} (see also~\cite[Proposition~2.7]{maret-ergodicity}) and the return implication follows immediately from Theorem~\ref{thm-intro}. The statement of Corollary~\ref{cor-intro-characterization-DT} can be refined depending on the peripheral data of $\rho$. The angle of rotation of the image of a peripheral curve is a called a \emph{peripheral angle} (Section~\ref{sec:genus-zero}). A peripheral angle is a number in $(0,2\pi)$. When the sum of all peripheral angles of $\rho$ (one for each puncture of $\surface$) is not an integer multiple of $2\pi$, then $\rho$ cannot be orthogonal. In that case, the statement of Corollary~\ref{cor-intro-characterization-DT} can be refined as follows: if the sum of the peripheral angles of $\rho$ is not an integer multiple of $2\pi$, then $\rho$ is a DT representation if and only if $\rho$ is totally elliptic.

In Section~\ref{sec:PSL(2,C)}, we will generalize the statement of Theorem~\ref{thm-intro} to surface group representations into $\PSL_2\C$---the complexification of $\PSL_2\R$. Maximal compact subgroups of $\PSL_2\C$ are conjugate to the projective unitary group $\PU(2)$ and surface group representations into a compact subgroup of $\PSL_2\C$ are accordingly called \emph{unitary representations}. In order to generalize Theorem~\ref{thm-intro} from $\psl$ to $\PSL_2\C$, we need to distinguish between irreducible and reducible representations. Recall that a representation $\rho\colon\pi_1\surface\to\PSL_2\C$ is \emph{reducible} if its image is conjugate to a subgroup made of upper triangular elements; it is \emph{irreducible} otherwise.\footnote{Despite the unfortunate linguistic resemblance, the notion of reducible representations and that of reduced representations (Definition~\ref{def:reduced-representation}) are different.}
\begin{thmx}[Propositions~\ref{prop-irreducible} \&~\ref{prop-reducible}]\label{thm-intro-complex}
Consider a surface $\surface$ of genus $g\geq 0$ with $n\geq 0$ punctures, and let $\rho\colon\pi_1\surface\to\PSL_2\C$ be a reduced totally elliptic representation.
\begin{itemize}
\item If $\rho$ is irreducible, then either $\rho$ is unitary or $\surface$ must be sphere with at least three punctures $(g=0$, $n\geq 3)$, and $\rho$ is conjugate to the composition of a DT representation with the inclusion $\PSL_2\R\subset \PSL_2\C$.
\item If $\rho$ is reducible and $g\geq 1$, then the image of $\rho$ is conjugate to a subgroup made of diagonal elements in $\PU(2)$, hence abelian.
\end{itemize}
\end{thmx}
The fact that totally elliptic representations into $\mathrm{PSL}_2\C$ must be unitary when $g\geq 1$ follows from the work of Gallo--Kapovich--Marden in the case $g\geq 3$ and $n=0$~\cite[Part~A]{GKM}, and in general from the subsequent work of Le Fils~\cite[Corollary~1.8]{thomas}. Like for Theorem~\ref{thm-intro}, the new contribution of this note concerns genus-$0$ surfaces.

Theorem~\ref{thm-intro-complex} leaves open the case of reducible totally elliptic representations into $\mathrm{PSL}_2\C$ for genus-0 surfaces. As it turns out, such representations are not necessarily unitary.
\begin{thmx}[Example~\ref{example-reducible}]\label{thm-intro-complex-reducible}
If $\surface$ is a sphere with at least three punctures $(g=0$, $n\geq 3)$, then there exist reduced totally elliptic representations $\pi_1\surface\to\mathrm{PSL}_2\C$ that are reducible but not unitary.
\end{thmx}
It is worth noting that the representations from Example~\ref{example-reducible} are indistinguishable from abelian unitary representations in the GIT character variety of $\surface$. In other words, the closures of their conjugacy classes in $\Hom(\pi_1\surface,\mathrm{PSL}_2\C)$ contain an abelian unitary representation.

\subsection{Motivations}
Characterizing a certain type of surface group representations by the images of \emph{simple} closed curves, as we did in Corollary~\ref{cor-intro-characterization-DT}, may seem baroque at first. Similar problems, however, have sparked substantial interest in recent decades. Bowditch famously raised the question of whether discrete and faithful representations of a closed surface group into $\psl$---also known as \emph{Fuchsian representations}---can be characterized by the property of mapping every simple closed curve to a hyperbolic element of $\psl$~\cite[Question~C]{bowditch}. In other words, are Fuchsian representations the only \emph{totally hyperbolic} closed surface group representations?
Bowditch's question has been answered positively for surfaces of genus $2$ by Marché--Wolff~\cite{marche-wolff, marche-wolff-2} and remains open in higher genera. 

The study of totally elliptic surface group representations is also motivated by the recent focus on finite mapping class group orbits in character varieties (see for instance~\cite{LLL, arnaud-sam, LT} and references therein), which led to a complete understanding of finite orbits for $\mathrm{SL}_2\C$-representations. Although not all finite orbits of conjugacy classes of surface group representations into $\mathrm{SL}_2\C$ consist of totally elliptic representations, many of them do (see~\cite{arnaud-sam} for more details). Totally elliptic representations can thus be viewed as a natural generalization of representations whose conjugacy classes lie in finite mapping class group orbits.

\subsection{Perspectives in higher rank}\label{sec:perspectives}
Totally elliptic representations also play an important role in the on-going effort to classify all compact connected components of relative character varieties for non-compact target Lie groups $G$.
Several families of compact components---all consisting entirely of conjugacy classes of totally elliptic representations---have been identified in the case where $G$ is a Hermitian Lie group by Tholozan--Toulisse~\cite{tholozan-toulisse} and Feng--Zhang~\cite{feng-zhang}. Wu showed that ``minimal energy systems'' also give rise to similar compact components~\cite{wu}. All these compact components arise specifically for punctured spheres.
The representations $\pi_1\surface\to G$ therein share a common property: the associated local systems support a \emph{universal variation of Hodge structure}, meaning that for every punctured Riemann sphere structure $X$ on $\surface$, there exists an equivariant (anti-)holomorphic map from the universal cover~$\widetilde X$ to the symmetric space $G/K$.
Unlike DT representations, which all have Zariski dense image in $\psl$, the Zariski closures of the images of the totally elliptic representations in those compact components are yet to be determined. Nevertheless, some of them must be Zariski dense, since Zariski dense representations form dense subsets of character varieties.

Moving beyond Hermitian Lie groups and $\mathrm{SL}_2\C$, it is interesting to notice that Lie groups may admit unbounded \emph{purely elliptic subgroups}---subgroups containing only elliptic elements.\footnote{A proof that every purely elliptic subgroup of $\mathrm{SL}_2\C$ is conjugate to a subgroup of $\mathrm{SU}(2)$ can be found in~\cite[Theorem~2.13.1]{complex-functions}. The analogous statement for $\mathrm{SL}_2\R$ and $\mathrm{SO}(2)$ is a direct consequence of Lemma~\ref{lem:commutator-hyperbolic}.} Bass gives an example of an unbounded purely elliptic subgroup of $\mathrm{SL}_3\C$ which is isomorphic to a rank-2 free group~\cite[Counterexample~1.10]{bass}. Bass' example, however, is not Zariski dense and, in fact, could not be as explained for instance in~\cite{prasad}. We also point out that when $\C$ is replaced by a non-Archimedean valued field $\mathbb{F}$, then the situation is similar to the case of $\mathrm{SL}_2\C$ as Parreau proved that a finitely generated subgroup of $\mathrm{SL}_n\mathbb{F}$ containing only elliptic elements is necessarily bounded~\cite[Théorème~1]{parreau}.

\subsection{Acknowledgements}
Many thanks to Samuel Bronstein and Nicolas Tholozan for enriching conversations on totally elliptic representations. I'm grateful to Bertrand Deroin for sparking my curiosity on the topic, and to Maxime Wolff for teaching me the elegant geometric argument used to prove Lemma~\ref{lem:commutator-hyperbolic}. I also extend my gratitude to my friends Jacques Audibert and Xenia Flamm for their insightful comments on an earlier draft of this paper and for pointing out to me Parreau's paper. I thank the anonymous referee for their relevant suggestions that improved the fluidity of the paper and considerably simplified the proof of Proposition~\ref{prop:totally-elliptic-implies-orthogonal}. This work was partially supported by the European Research Council (ERC) under the European Union’s Horizon 2020 research and innovation program (Grant agreement No.~101096550).

\section{Totally elliptic representations into $\psl$}\label{sec:proof-of-main-theorem}
\subsection{Overview}
This section contains the proof of Theorem~\ref{thm-intro}. We start with the simpler case of surfaces with genus $g\geq 1$ and show that, in this case, totally elliptic representations are orthogonal (Section~\ref{sec:positive-genus}). 
We then recall the definition of DT representations (Section~\ref{sec:genus-zero}) and prove the main result of this note: namely, that any totally elliptic genus-0 surface group representation into $\mathrm{PSL}_2\R$ that is not orthogonal must be a DT representation.

\subsection{Positive genus}\label{sec:positive-genus}
According to Definition~\ref{def:totally-elliptic}, the elliptic elements of $\psl$ fall into two categories: the regular elliptic elements and the identity. Regular elliptic elements are characterized by a rotation angle in $(0,2\pi)$ and a unique fixed point in the hyperbolic plane $\HH$. Each of them is contained in a unique maximal compact subgroup, a conjugate of $\mathrm{PSO}(2)$ that coincides with their centralizer. The following standard fact about commutators in $\mathrm{PSL}_2\R$ will be key to the argument.
\begin{lem}[{see e.g.~\cite[Theorem~7.39.2]{beardon}}]\label{lem:commutator-hyperbolic}
If $A\in\psl$ is regular elliptic and $X\in\psl$ is any element, then the commutator
\[
[A,X]=AXA^{-1}X^{-1}
\] 
is either trivial or hyperbolic. Furthermore, the commutator is trivial if and only if~$X$ is elliptic and fixes the same point as $A$.
\end{lem}
\begin{proof}
First, if $[A,X]=1$, then $X$ belongs to the centralizer of $A$, which means that $X$ is elliptic and fixes the unique fixed point of $A$. Now, if $[A,X]\neq 1$ and $o$ denotes the unique fixed point of $A$, then $XA^{-1}X^{-1}$ is regular elliptic with fixed point $X.o\neq o$. Note that if $\alpha$ is the rotation angle of $A$, then the rotation angle of $XA^{-1}X^{-1}$ is $2\pi-\alpha$. Let $\ell$ denote the unique geodesic line in~$\HH$ through $o$ and $X.o$. Similarly, let $\ell_1$ and $\ell_2$ denote the geodesic lines obtained by rotating~$\ell$ anti-clockwise by an angle $\alpha/2$ around $o$, respectively $X.o$.
\begin{center}
\begin{tikzpicture}[scale=2]
    \draw[gray] (1,0) to (-1,0);
    \node[anchor=south] at (0,0) {\small $\ell$};
    \node[anchor=south] at (0,0.3) {\small $\ell_1$};
    \node[anchor=north] at (0,-.3) {\small $\ell_2$};
    \node[anchor=north] at (-.5,0) {\small $o$};
    \node[anchor=south] at (.5,0) {\small $X.o$};
    \tkzDefPoint(0,0){O}
    \tkzDefPoint(1,0){A}
    \tkzDrawCircle(O,A)
    \tkzDefPoint(0.5,0){Xo}
    \tkzDefPoint(-0.5,0){o}
    \tkzDefPoint(0,0.3){w1}
    \tkzDefPoint(0,-0.3){w2}
    \tkzClipCircle(O,A)
    \tkzDefCircle[orthogonal through=Xo and w2](O,A)
    \tkzGetPoints{G}{g}
    \tkzDrawCircle(G,g)
    \tkzDefCircle[orthogonal through=o and w1](O,A)
    \tkzGetPoints{H}{h}
    \tkzDrawCircle(H,h)
    \tkzDrawPoints[color=black,fill=black,size=4](Xo,o)
    
    \draw[apricot] (-.3,-.3) ++(0, .3) arc (0:18:.3) node[anchor=south]{\small $\alpha/2$};
    \draw[apricot] (.3,-.3) ++(0, .3) arc (180:198:.3) node[anchor=north]{\small $\alpha/2$};
\end{tikzpicture}
\end{center}
The orientation-reversing reflections of $\HH$ through the lines $\ell$, $\ell_1$, and $\ell_2$ are denoted by $\sigma$, $\sigma_1$, and $\sigma_2$, respectively. With this notation, we can write $A=\sigma_1\sigma$ and $XA^{-1}X^{-1}=\sigma\sigma_2$, so that $[A,X]=\sigma_1\sigma_2$. We conclude that $[A,X]$ preserves the common orthogonal geodesic line to $\ell_1$ and $\ell_2$, and is therefore hyperbolic.
\end{proof}
The next two statements are well-known properties of simple closed curves on surfaces. More details can be found in~\cite[Section~1.3]{primer-mcg}.
\begin{fact}\label{fact:commutator-scc-is-scc}
Let $\alpha$ and $\beta$ be two simple closed loops on a surface $\surface$ intersecting once at $x$ and such that the geometric intersection number of their free homotopy classes is one. We denote by $a$ and $b$ the fundamental group elements with basepoint $x$ defined by $\alpha$ and $\beta$. If $[a,b]$ is a non-trivial fundamental group element (which is always the case if $\surface$ is not a closed torus), then it represents a simple closed curve on $\surface$.
\end{fact}
Recall that a simple closed curve $a$ on a surface $\Sigma$ is \emph{separating} if cutting $\Sigma$ along any simple loop representative of $a$ disconnects $\Sigma$ and it is \emph{non-separating} otherwise. For instance, peripheral curves (simple closed curves homotopic to a puncture) are separating.
\begin{fact}\label{fact:MCG-acts-transitively-on-non-separating-scc}
The mapping class group of $\Sigma$ acts transitively on the set of non-separating simple closed curves.
\end{fact}
We will work with a \emph{geometric generating family} for $\pi_1\surface$ which is any set of $2g+n$ generators $a_1,\ldots,a_g$, $b_1,\ldots,b_g$, and $c_1,\ldots,c_n$ satisfying the unique relation
\[
\prod_{i=1}^g[a_i,b_i]=\prod_{i=1}^n c_i.
\]
A possible choice of geometric generators is given by the homotopy classes of the loops illustrated on Figure~\ref{fig:geometric-generators}.
 
\begin{center}
\begin{figure}[!ht]
\begin{tikzpicture}[scale=.8]
\node[anchor=south west, inner sep=0mm]{\includegraphics[scale=.64]{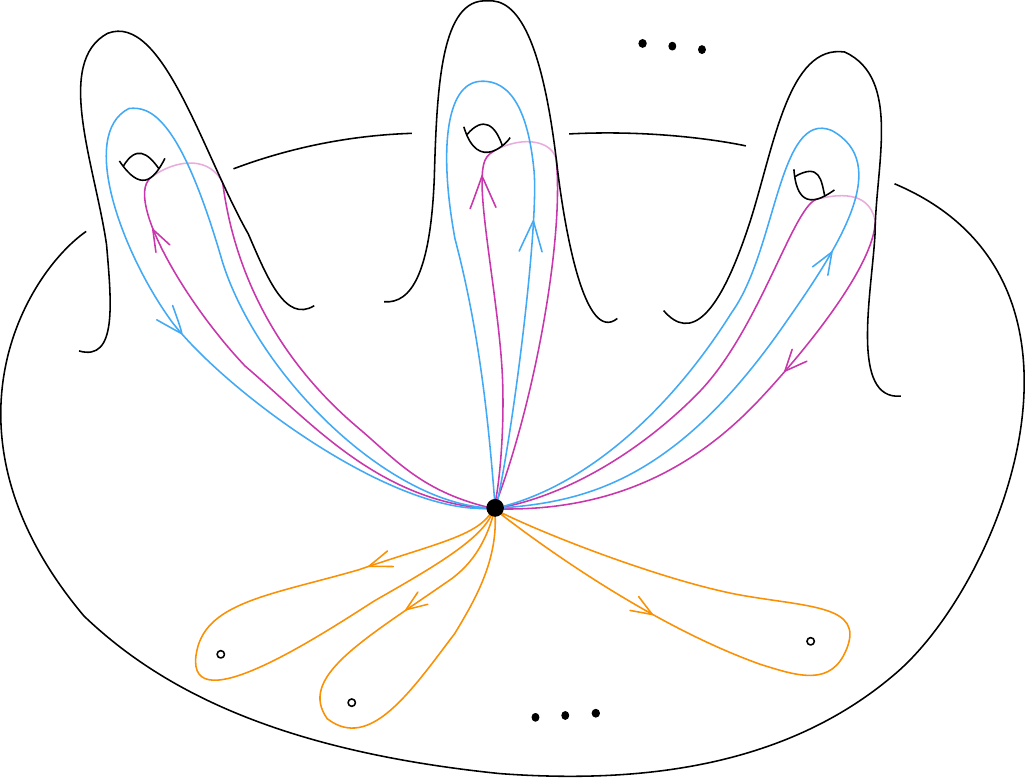}};
\draw[mauve] (10.55,5) node[anchor=south west, inner sep=0mm] {\Large $a_1$};
\draw[mauve] (2.1,7.5) node[anchor=south west, inner sep=0mm] {\Large $a_g$};
\draw[mauve] (7.45,5.5) node[anchor=south west, inner sep=0mm] {\Large $a_{g-1}$};
\draw[sky] (10.9,7.2) node[anchor=south west, inner sep=0mm] {\Large $b_1$};
\draw[sky] (2.15,5.5) node[anchor=south west, inner sep=0mm] {\Large $b_g$};
\draw[sky] (5.55,5.5) node[anchor=south west, inner sep=0mm] {\Large $b_{g-1}$};
\draw[apricot] (8.2,1.9) node[anchor=south west, inner sep=0mm] {\Large $c_1$};
\draw[apricot] (6.05,1.4) node[anchor=south west, inner sep=0mm] {\Large $c_{n-1}$};
\draw[apricot] (4.5,2.95) node[anchor=south west, inner sep=0mm] {\Large $c_n$};
\end{tikzpicture}
\caption{A system of geometric generators for a surface $\surface$ of genus $g$ with $n$ punctures.}\label{fig:geometric-generators}
\end{figure}
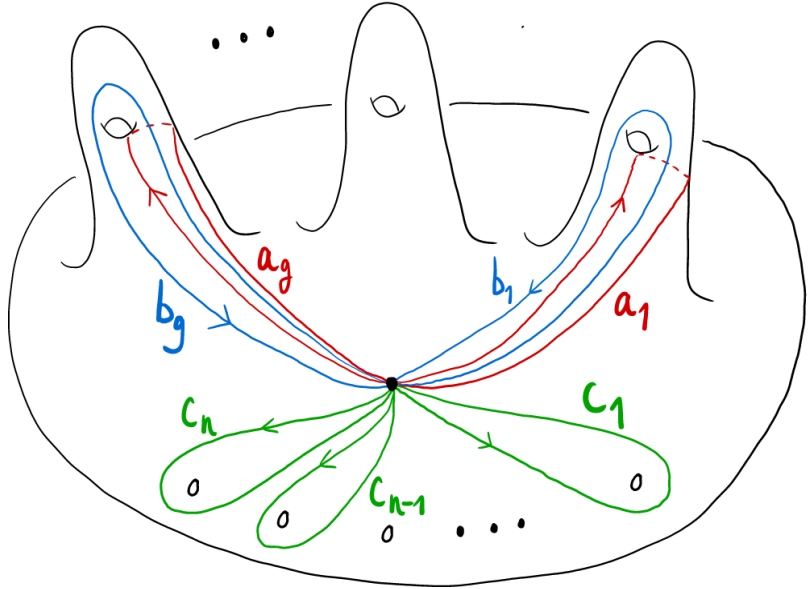
\end{center}
\begin{rmk}\label{rem:concatenation-convention}
We adopt the convention of concatenating loops by starting with the rightmost loop. More precisely, if $\gamma_1$ and $\gamma_2$ are two loops on $\surface$ based at the same point, then their concatenation $\gamma_1\gamma_2$ is the loop obtained by going along $\gamma_2$ first, and then along $\gamma_1$. This rule applies for fundamental group multiplications too.
\end{rmk}

Note that the two following sets of $2g+n$ fundamental group elements are also generating sets for $\pi_1\surface$:
\begin{align}
&\{a_1,\ldots,a_g,b_1,\ldots,b_g,a_1^{-1}c_1,\ldots,a_1^{-1}c_n\}, \label{eq:generating-set-non-separating} \\
&\{a_1^{-1}b_1, a_2^{-1}b_1,\ldots, a_g^{-1}b_1, b_1,b_2b_1,\ldots,b_gb_1,b_1c_1,\ldots,b_1c_n\}. \label{eq:generating-set-interesct-a_1-once}
\end{align}
The set~\eqref{eq:generating-set-non-separating} consists of fundamental group elements representing non-separating simple closed curves when $g\geq 1$. Furthermore, every element $x$ of the set~\eqref{eq:generating-set-interesct-a_1-once} can be represented by a simple closed loop that intersects $a_1$ only at the basepoint, implying that $[x,a_1]$ represents a simple closed curve by Fact~\ref{fact:commutator-scc-is-scc}.

\begin{prop}\label{prop:totally-elliptic-implies-orthogonal}
Assume that $\surface$ is a surface of genus $g\geq 1$ with an arbitrary number of punctures. If $\rho\colon\pi_1\surface\to\psl$ is a totally elliptic representation, then $\rho$ is an orthogonal representation.
\end{prop}
\begin{proof}
The conclusion holds for the trivial representation. If $\rho$ is not the trivial representation, then it maps an element of the generating set~\eqref{eq:generating-set-non-separating} to a non-trivial element of $\psl$. Since all the curves of~\eqref{eq:generating-set-non-separating} are non-separating, up to changing the system of geometric generators of $\pi_1\surface$, we may assume by Fact~\ref{fact:MCG-acts-transitively-on-non-separating-scc} that $\rho(a_1)\neq \id$. This means that $\rho(a_1)$ is regular elliptic because $\rho$ is totally elliptic. Up to conjugating $\rho$, we may assume that $\rho(a_1)\in\mathrm{PSO(2)}$. In particular, $[X,\rho(a_1)]$ is elliptic if and only if $X\in\mathrm{PSO}(2)$ by Lemma~\ref{lem:commutator-hyperbolic}. Since $[x,a_1]$ represents a simple closed curve for every $x$ in the generating set~\eqref{eq:generating-set-interesct-a_1-once}, we conclude that $\rho(x)\in\mathrm{PSO}(2)$ for every such $x$, and thus $\rho$ is orthogonal.
\end{proof}

\subsection{Genus zero}\label{sec:genus-zero}
First, note that if $\surface$ is a sphere with at most two punctures, then every totally elliptic surface group representation $\pi_1\surface\to\psl$ is orthogonal. So we will assume that $\surface$ is a sphere with a set $\mathcal{P}$ of $n\geq 3$ punctures. For simplicity, we want to avoid the case where peripheral curves have trivial images.

\begin{defn}\label{def:reduced-representation}
A surface group representation $\rho\colon\pi_1\surface\to\PSL_2\C$ is said to be \emph{reduced} if every peripheral curve has non-trivial image.
\end{defn}
Every non-reduced surface group representation $\rho\colon\pi_1\surface\to\PSL_2\C$ factorizes through a reduced representation $\rho'\colon\pi_1\surface'\to\PSL_2\C$, where~$\surface'$ is obtained from $\surface$ by filling in the punctures whose corresponding peripheral curves are in the kernel of~$\rho$.
Note that if $\rho$ is totally elliptic, then so is $\rho'$.

For every angle vector $\alpha\in (0,2\pi)^{\mathcal{P}}$, the subspace of the character variety $\Rep(\surface,\psl)$ consisting of all conjugacy classes of representations mapping every peripheral curve around a puncture $p\in\mathcal{P}$ to regular elliptic elements of rotation angle $\alpha_p$ is called \emph{$\alpha$-relative character variety} and is denoted by
\[
\Rep_\alpha(\surface,\psl)\subset \Rep(\surface,\psl).
\]
The numbers~$\alpha_p$ are \emph{peripheral angles} and we will denote their sum by
\[
|\alpha|=\sum_{p\in\mathcal{P}}\alpha_p.
\]
By definition, a reduced totally elliptic representation $\rho\colon\pi_1\surface\to\psl$ maps every peripheral curve on~$\surface$ to a regular elliptic element of $\psl$, and its conjugacy class $[\rho]$ therefore lies in an $\alpha$-relative character variety.

In contrast to the positive genus case (see Proposition~\ref{prop:totally-elliptic-implies-orthogonal}), there exist totally elliptic genus-0 surface group representations with Zariski dense image in $\psl$. Their existence was first attested by Benedetto--Goldman in the case $n=4$, and later by Deroin--Tholozan for all $n\geq 4$.

\begin{thm}[\cite{benedetto-goldman, deroin-tholozan}]\label{thm:DT}
If $|\alpha|<2\pi$ or $|\alpha|>2\pi (n-1)$, then $\Rep_\alpha(\surface,\psl)$ contains a compact connected component isomorphic to $\mathbb{CP}^{n-3}$ and entirely made of conjugacy classes of totally elliptic representations with Zariski dense image. Furthermore, these representations map every simple closed curve to a regular elliptic element.
\end{thm}
\begin{defn}\label{def:DT-representation}
A representation whose conjugacy class belongs to one of the compact components of Theorem~\ref{thm:DT} will be called a \emph{DT representation}. The components themselves will be referred to as \emph{DT components}. 
\end{defn}

DT representations can be characterized by their Toledo number in the sense of Burger--Iozzi--Wienhard~\cite{BIW}. Namely, according to Deroin--Tholozan's original definition, a point in $\Rep_\alpha(\surface,\psl)$ belongs to a DT component if and only if its Toledo number is
\begin{equation}\label{eq:toledo}
\begin{cases}  
1-\frac{|\alpha|}{2\pi} & \text{when } |\alpha|<2\pi,\\
(n-1)-\frac{|\alpha|}{2\pi} & \text{when }|\alpha|>2\pi (n-1).
\end{cases}
\end{equation}
\begin{rmk}
Observe that the Toledo number of a DT representation is always contained in $(-1,1)\setminus \{0\}$ and can be arbitrarily close to $0$ when $\vert\alpha\vert$ approaches $2\pi$ or $2\pi(n-1)$. 
In other words, DT representations exhibit nearly minimal Toledo number (in absolute value) among all points of $\Rep(\surface,\psl)$.
\end{rmk}

The topology of relative $\psl$ character varieties has been studied by Mondello. Among other things, he characterized all their compact components, proving that the only compact connected components that are not isolated points are DT components.

\begin{thm}[{\cite[Corollary~4.17]{mondello}}]\label{thm:mondello}
The $\alpha$-relative character variety $\Rep_\alpha(\surface,\psl)$ has at most one compact connected component, and it exists only in the following cases:
\begin{itemize}
\item $|\alpha|\in 2\pi\Z$ and the compact component is an isolated point corresponding to an orthogonal representation.
\item $|\alpha|\in (0,2\pi)\cup (2\pi(n-1),2\pi n)$ and the compact component is one of the DT components from Theorem~\ref{thm:DT}.
\end{itemize}
\end{thm}

To conclude the proof of Theorem~\ref{thm-intro}, we now show that if $[\rho]\in \Rep_\alpha(\surface,\psl)$ is the conjugacy class of a reduced totally elliptic representation, then $\alpha$ must be one of the angle vectors from Theorem~\ref{thm:mondello}, and $[\rho]$ must belong to the unique compact connected component of $\Rep_\alpha(\surface,\psl)$. This is a refinement of the following result by Deroin--Tholozan, who observed that conjugacy classes of totally elliptic representations remain in a bounded region of the character variety.
\begin{prop}[{\cite[Proposition~2.5]{deroin-tholozan}}]\label{prop:totally-elliptic-bounded}
The subset of all conjugacy classes of totally elliptic representations inside the character variety $\Rep(\surface,\psl)$ is bounded (i.e.~contained in a compact subset).
\end{prop}
For the reader's convenience, we briefly recall the argument of Deroin--Tholozan. First, consider the closure of the subset of totally elliptic representations inside $\Hom(\pi_1\surface, \psl)$ and denote it by $\overline{\mathcal{TE}}$. The set $\overline{\mathcal{TE}}$ is contained in the collection of all representations in $\Hom(\pi_1\surface, \psl)$ that map every simple closed curve on $\surface$ to either an elliptic or a parabolic element of $\psl$. All the representations in $\overline{\mathcal{TE}}$ are dominated (in the sense of~\cite{GKW, DT-domination}) by every Fuchsian representation in $\Hom(\pi_1\surface, \psl)$ by a result of Guéritaud--Kassel~\cite{GK}. So, if we fix an arbitrary Fuchsian representation $j\colon\pi_1\surface\to\psl$, it will dominate every representation in a sequence $(\rho_k)\subset \overline{\mathcal{TE}}$. Domination means that to every $\rho_k$ corresponds a $(j,\rho_k)$-equivariant 1-Lipschitz map $\HH\to\HH$. The existence of a converging subseqeunce for $(\rho_k)$ now follows from the Arzelà--Ascoli theorem (see~\cite{deroin-tholozan} for more details). This argument shows that $\overline{\mathcal{TE}}\subset \Hom(\pi_1\surface, \psl)$ is sequentially compact, hence compact because $\Hom(\pi_1\surface, \psl)$ is a metrizable space. Now, if $p$ denotes the quotient map from $\Hom(\pi_1\surface, \psl)$ to the non-Hausdorff quotient $\Rep(\surface,\psl)$, then the subset of all conjugacy classes of totally elliptic representations in $\Rep(\surface,\psl)$ is contained in the compact set~$p(\overline{\mathcal{TE}})$.

\subsubsection{The case $n=3$}\label{sec:n=3}
A $3$-punctured sphere $\surface$ only has three simple closed curves: its three peripheral curves. Any point in an $\alpha$-relative character variety $\Rep_\alpha(\surface,\psl)$ is thus automatically totally elliptic. 
Furthermore, if $c_1,c_2,c_3$ are geometric generators of $\pi_1\Sigma$ satisfying $c_1c_2c_3=1$, then every reduced totally elliptic representation $\rho$ corresponds to a hyperbolic triangle whose vertices $(C_1,C_2,C_3)$ are the fixed points of $\rho(c_1), \rho(c_2), \rho(c_3)$. The possible configurations of $(C_1,C_2,C_3)$ depend on the peripheral angles $\alpha=(\alpha_1,\alpha_2,\alpha_3)\in (0,2\pi)^3$ of $\rho$ (see e.g.~\cite[Section~3.1 and Table~1]{action-angle}):
\begin{itemize}
\item $|\alpha|\in\{2\pi,4\pi\}$ and the triangle is degenerate with $C_1=C_2=C_3$.
\item $|\alpha|<2\pi$ and the triangle $(C_1,C_2,C_3)$ is anti-clockwise oriented with interior angles given by $(\alpha_1/2,\alpha_2/2,\alpha_3/2)$.
\item $|\alpha|>4\pi$ and the triangle $(C_1,C_2,C_3)$ is clockwise oriented with interior angles given by $(\pi-\alpha_1/2,\pi-\alpha_2/2,\pi-\alpha_3/2)$.
\end{itemize}
If $|\alpha|\in\{2\pi,4\pi\}$, then $\rho$ is orthogonal. When $|\alpha|<2\pi$ or when $|\alpha|>4\pi$, then $\rho$ is a DT representation. Finally, if $|\alpha|\in (2\pi, 4\pi)$, then $\Rep_\alpha(\surface,\psl)$ is empty. Altogether, this proves Theorem~\ref{thm-intro} for 3-punctured spheres.

\subsubsection{The case $n=4$}\label{sec:n=4}
Throughout this section, $\surface$ will denote a 4-punctured sphere. We start by considering the case of representations with a non-peripheral simple closed curve in their kernel.
\begin{lem}\label{lem:scc-in-kernel-n=4}
If $\rho\colon\pi_1\surface\to\psl$ is a totally elliptic representation with a non-peripheral simple closed curve in its kernel, then $\rho$ is orthogonal.
\end{lem}
\begin{proof}
We pick a geometric presentation of $\pi_1\surface$ with generators $c_1,c_2,c_3,c_4$ satisfying $c_1c_2c_3c_4=1$, as on Figure~\ref{fig:geometric-generators}, and such that $\rho(c_1c_2)=\id$. This means that there exist $X,Y\in\psl$ such that $\rho(c_1)=X^{-1}$, $\rho(c_2)=X$, $\rho(c_3)=Y$, and $\rho(c_4)=Y^{-1}$. Observe that $c_1c_3^{-1}c_2c_3$ represents a simple closed curve, so $\rho(c_1c_3^{-1}c_2c_3)=[X^{-1},Y^{-1}]$ is elliptic. By Lemma~\ref{lem:commutator-hyperbolic}, this is possible only if $[X^{-1},Y^{-1}]=1$, proving that $\rho$ is orthogonal.
\end{proof}

We now consider reduced totally elliptic representations $\rho\colon\pi_1\surface\to\psl$ with no simple closed curves in their kernel. 
The conjugacy class $[\rho]$ belongs to an $\alpha$-relative character variety $\Rep_\alpha(\surface,\psl)$ for some angle vector $\alpha\in(0,2\pi)^4$.

The case $|\alpha|<2\pi$ or $|\alpha|>6\pi$ was already treated in~\cite{maret-ergodicity}. It turns out that mapping a single non-peripheral closed curve to an elliptic element is enough to be a DT representation in that case.
\begin{lem}[{\cite[Remark~2.8]{maret-ergodicity}}]\label{lem:4-punctured-small-big-angles}
Assume $\alpha\in (0,2\pi)^4$ satisfies $|\alpha|<2\pi$ or $|\alpha|>6\pi$ and let $[\rho]\in \Rep_\alpha(\surface,\psl)$. If $\rho$ maps some non-peripheral simple closed curve to a regular elliptic element, then $\rho$ is a DT representation. In particular, every such totally elliptic representation is a DT representation.
\end{lem}
The proof of Lemma~\ref{lem:4-punctured-small-big-angles} given in~\cite{maret-ergodicity} uses the characterization of DT representations in terms of Toledo number (given by~\eqref{eq:toledo}) and the additive properties of the Toledo number. 

When $|\alpha|\in [2\pi,6\pi]$, the situation is more subtle, because there exist representations that map some non-peripheral simple closed curves on $\surface$ to elliptic elements, but which are not totally elliptic. An explicit example is provided by Palesi in~\cite[Section~5.4]{palesi}. We will nevertheless prove that if $|\alpha|\in [2\pi,6\pi]$ and $[\rho]\in \Rep_\alpha(\surface,\psl)$ is the conjugacy class of a totally elliptic representation, then actually $|\alpha|\in\{2\pi,4\pi,6\pi\}$ and $\rho$ is orthogonal. We will use an argument of dynamical flavour that turns out to also cover the case of Lemma~\ref{lem:4-punctured-small-big-angles}.

The pure mapping class group of $\surface$ acts on $\Rep_\alpha(\surface,\psl)$ for any angle vector $\alpha\in (0,2\pi)^4$ (the action is described in details in~\cite{maret-ergodicity}). If $\rho$ is totally elliptic, then every point in the orbit of $[\rho]$ is the conjugacy class of a totally elliptic representation. The orbit of $[\rho]$ is therefore bounded by Proposition~\ref{prop:totally-elliptic-bounded}. There are now two possibilities.

First, the orbit of $[\rho]$ may be finite. It turns out that finite orbits have been classified by Lisovyy--Tykhyy~\cite{LT}. Their classification shows that any finite orbit in an $\alpha$-relative character variety $\Rep_\alpha(\surface,\psl)$ either belongs to a DT component, or is an isolated point, hence the conjugacy class of an orthogonal representation by Theorem~\ref{thm:mondello}.

Secondly, we shall consider the case where the orbit of $[\rho]$ is an infinite bounded set. Infinite bounded orbits, in the context of 4-punctured spheres, have been studied by Cantat--Loray. A consequence of their work is the following statement.\footnote{There is a small subtlety here. Cantat--Loray study the pure mapping class group action on the real points of the complex GIT quotient of $\Hom(\pi_1\surface, \mathrm{SL}_2\C)$ by $\mathrm{SL}_2\C$ and not on the topological quotient of $\Hom(\pi_1\surface, \psl)$ by $\psl$ like we do. However, since all the representations whose conjugacy class lies in an $\alpha$-relative character variety $\Rep_\alpha(\surface,\psl)$ are reductive, Theorem~\ref{thm:cantat-loray} is indeed a consequence of Cantat--Loray's results.}

\begin{thm}[{\cite[Theorem~C]{cantat-loray}}]\label{thm:cantat-loray}
If an  $\alpha$-relative character variety $\Rep_\alpha(\surface,\psl)$ contains an infinite bounded orbit, then $|\alpha| < 2\pi$ or $|\alpha| > 6\pi$ and the orbit is dense in the DT component of $\Rep_\alpha(\surface,\psl)$.
\end{thm}

Consequently, no matter if the orbit of $[\rho]$ is finite or infinite, $\rho$ is either a DT representation or orthogonal. This proves Theorem~\ref{thm-intro} for 4-punctured spheres.

\subsubsection{The case $n\geq 5$}\label{sec:n-geq-5}
Let now $\surface$ denote a sphere with $n\geq 5$ punctures and $\rho\colon\pi_1\surface\to\psl$ be a reduced totally elliptic representation. As before, we start with the case of representations with a non-peripheral simple closed curve in their kernel.

\begin{lem}\label{lem:scc-in-kernel}
If $\rho\colon\pi_1\surface\to\psl$ is a reduced totally elliptic representation with a non-peripheral simple closed curve in its kernel, then $\rho$ is orthogonal.    
\end{lem}
\begin{proof}
There are geometric generators $c_1,\ldots, c_n$ of $\pi_1\surface$ satisfying $c_1\cdots c_n=1$ (as on Figure~\ref{fig:geometric-generators}) such that the simple closed curve in the kernel of $\rho$ is represented by $c_1\cdots c_{i+1}$ for some $i\in\{1,\ldots,n-3\}$. Hence, $\rho(c_{i+1})=\rho(c_1\cdots c_i)^{-1}$ and $\rho(c_{i+2})=\rho(c_{i+3}\cdots c_n)^{-1}$. Observe that $(c_1\cdots c_i) c_{i+2}^{-1}c_{i+1}c_{i+2}$ represents a simple closed curve (it corresponds to the image of $c_1\cdots c_{i+1}$ by the Dehn twist along $c_{i+1}c_{i+2}$). Therefore, $\rho\big((c_1\cdots c_i) c_{i+2}^{-1}c_{i+1}c_{i+2}\big)=\big[\rho(c_{i+1}^{-1}),\rho(c_{i+2}^{-1})\big]$ is elliptic and thus $\big[\rho(c_{i+1}),\rho(c_{i+2})\big]=1$ by Lemma~\ref{lem:commutator-hyperbolic}. So, up to conjugating $\rho$, we may assume that $\rho(c_{i+1})$ and $\rho(c_{i+2})$ both belong to $\PSO(2)$.

Observe further that 
\[
(c_1\cdots c_i)c_j^{-1}c_{i+1}c_j\quad \text{and}\quad (c_{i+3}\cdots c_n)^{-1}c_kc_{i+2}^{-1}c_k^{-1}
\]
represent simple closed curves for every $j\geq i+2$ and $k\leq i+1$ (they are the images of $c_1\cdots c_{i+1}$ and $(c_{i+2}\cdots c_n)^{-1}$ by the Dehn twists along $c_{i+1}c_j$ and $c_kc_{i+2}$, respectively). Therefore, 
\[
\rho\big((c_1\cdots c_i)c_j^{-1}c_{i+1}c_j\big)=\big[\rho(c_{i+1})^{-1},\rho(c_j)^{-1}\big]\text{ and }\rho\big((c_{i+3}\cdots c_n)^{-1}c_jc_{i+2}^{-1}c_j^{-1}\big)=\big[\rho(c_{i+2}),\rho(c_j)\big]
\]
are both elliptic, hence trivial by Lemma~\ref{lem:commutator-hyperbolic}. Since $\rho$ is reduced, we have $\rho(c_{i+1})\neq \id$ and $\rho(c_{i+2})\neq \id$, implying that $\rho(c_j)\in\PSO(2)$ for every $j$, which proves that $\rho$ is orthogonal.
\end{proof}

From now on, we will assume that $\rho$ is a reduced totally elliptic representation with no simple closed curves in its kernel. 
To such a representation corresponds a \emph{chain of triangles} in $\HH$. It is constructed from a \emph{chained pants decomposition} $\mathcal{B}$ of $\surface$, that is a pants decomposition of $\surface$ in which each pair of pants contains at least one puncture. For each chained pants decomposition $\mathcal{B}$, there is a geometric presentation of $\pi_1\surface$ with generators $c_1,\ldots,c_n$ satisfying $c_1\cdots c_n=1$, such that the $n-3$ pants curves of $\mathcal{B}$ are represented by the fundamental group elements $b_i=(c_1\cdots c_{i+1})^{-1}$ for $i=1,\ldots, n-3$.\footnote{Such a presentation can be constructed by inducting on $n\geq 3$ and using Van Kampen's Theorem, see e.g.~\cite[Appendix~B]{action-angle} for details.}
Since $\rho$ is reduced and has no simple closed curves in its kernel, all the elliptic elements $\rho(c_1),\ldots,\rho(c_n)$ and $\rho(b_1),\ldots,\rho(b_{n-3})$ are regular and have a unique fixed point in $\HH$, which we denote by $C_1,\ldots,C_n$ and $B_1,\ldots, B_{n-3}$. The \emph{$\mathcal{B}$-triangle chain of $\rho$} consists of the $n-2$ triangles
\[
(C_1,C_2,B_1), (B_1,C_3,B_2),\ldots, (B_{n-4},C_{n-2},B_{n-3}), (B_{n-3},C_{n-1},C_n).
\]
For instance, when $n=6$, the $\mathcal{B}$-triangle chain of $\rho$ may look like this.
\begin{center}
\begin{tikzpicture}[font=\sffamily,decoration={
    markings,
    mark=at position 1 with {\arrow{>}}}]]
    
\node[anchor=south west,inner sep=0] at (0,0) {\includegraphics[width=9cm]{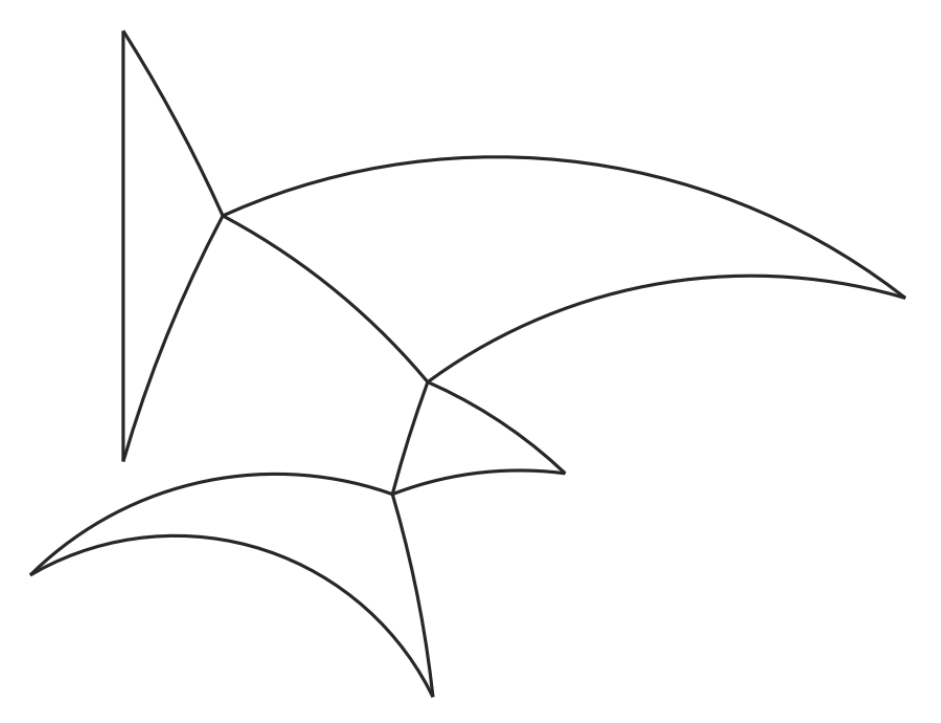}};

\begin{scope}
\fill (1.19,2.53) circle (0.07) node[left]{$C_1$};
\fill (1.19,6.62) circle (0.07) node[left]{$C_2$};
\fill (8.65,4.08) circle (0.07) node[below right]{$C_3$};
\fill (5.4,2.4) circle (0.07) node[below right]{$C_4$};
\fill (4.14,0.28) circle (0.07) node[below right]{$C_5$};
\fill (.31,1.43) circle (0.07) node[left]{$C_6$};
\end{scope}

\begin{scope}[apricot]
\fill (2.15,4.86) circle (0.07);
\fill (2.25,5.2) node{$B_1$};
\fill (4.1,3.27) circle (0.07) node[left]{$B_2$};
\fill (3.76,2.2) circle (0.07);
\fill (4.05,2.05) node{$B_3$};
\end{scope}
\end{tikzpicture}
\end{center}
For more details on triangle chains, the reader is referred to~\cite{action-angle}.

Each triangle can be either degenerate to a single point, or clockwise or anti-clockwise oriented, as we explained in Section~\ref{sec:n=3}. It turns out that being a DT representation can be characterized by the orientations of the triangles.
\begin{prop}[{\cite[Lemma~3.5]{action-angle}}]\label{prop:DT-if-triangle-coherently-oriented}
If there exists a chained pants decomposition $\mathcal{B}$ of $\surface$ such that the $\mathcal{B}$-triangle chain of $\rho$ contains at least one non-degenerate triangle and all non-degenerate triangles have the same orientation, then $\rho$ is a DT representation. Conversely, if $\rho$ is a DT representation, then for every chained pants decomposition $\mathcal{B}$ of $\surface$, all non-degenerate triangles in the $\mathcal{B}$-triangle chain of $\rho$ have the same orientation.
\end{prop}

If the $\mathcal{B}$-triangle chain of $\rho$ contains only degenerate triangles, then $\rho$ is orthogonal. Otherwise, if at least one triangle is non-degenerate, then the $\rho(c_i)$ do not all fix the same point and $\rho$ is not orthogonal. In the latter case, we can replace $\mathcal{B}$ by a new chained pants decomposition and find a compatible geometric presentation of $\pi_1\surface$ such that the new triangle chain of $\rho$ only consists of non-degenerate triangles. This new pants decomposition can be constructed using the same algorithm as in~\cite[Proposition~2]{aaron-arnaud} and the compatible geometric presentation of $\pi_1\surface$ is obtained by the same reasoning as in~\cite[Proposition~3]{aaron-arnaud}. From now on, we will therefore assume that all the triangles in the $\mathcal{B}$-triangle chain of $\rho$ are non-degenerate.

We denote the rotation angles of the images of the pants curves $b_1,\ldots,b_{n-3}$ under $\rho$ by $\beta_1,\ldots,\beta_{n-3}\in(0,2\pi)$. %Recall that $\rho(b_1),\ldots,\rho(b_{n-3})$ are all regular elliptic because we are assuming $\rho$ to be regularly totally elliptic. 

\begin{prop}\label{prop:totally-elliptic-implies-DT}
Under our assumptions that $\rho$ is a reduced totally elliptic representation with no simple closed curves in its kernel and that every triangle in its $\mathcal{B}$-triangle chain is non-degenerate, then all the triangles must have the same orientation. Hence, $\rho$ is a DT representation by Proposition~\ref{prop:DT-if-triangle-coherently-oriented}.
\end{prop}
\begin{proof}
For each $i=1,\ldots,n-3$, we consider the 4-punctured spheres $\surface^{(i)}\subset \surface$ with peripheral curves $(b_{i-1})^{-1},c_{i+1},c_{i+2},b_{i+1}$.
\begin{center}
\vspace{2mm}
\begin{tikzpicture}[scale=1.1, decoration={
    markings,
    mark=at position 0.6 with {\arrow{>}}}]
  \draw[postaction={decorate}, black!40] (0,-.5) arc(-90:-270: .25 and .5);
  \draw[black!40] (0,.5) arc(90:-90: .25 and .5);
  \draw[postaction={decorate}] (2,.5) arc(90:270: .25 and .5) node[midway, left]{$b_{i-1}$};
  \draw (2,.5) arc(90:-90: .25 and .5);
  \draw[apricot, postaction={decorate}] (4,.5) arc(90:270: .25 and .5) node[midway, left]{$b_i$};
  \draw[lightapricot] (4,.5) arc(90:-90: .25 and .5);
  \draw[postaction={decorate}] (6,.5) arc(90:270: .25 and .5) node[midway, left]{$b_{i+1}$};
  \draw (6,.5) arc(90:-90: .25 and .5);
  \draw[postaction={decorate}, black!40] (8,.5) arc(90:270: .25 and .5);
  \draw[black!40] (8,.5) arc(90:-90: .25 and .5);
  
  \draw[black!40](.5,1) arc(180:0: .5 and .25)node[midway, above]{$c_{i}$};
  \draw[postaction={decorate, black!40}, black!40] (.5,1) arc(-180:0: .5 and .25);
  \draw (2.5,1) arc(180:0: .5 and .25)node[midway, above]{$c_{i+1}$};
  \draw[postaction={decorate}] (2.5,1) arc(-180:0: .5 and .25);
  \draw (4.5,1) arc(180:0: .5 and .25)node[midway, above]{$c_{i+2}$};
  \draw[postaction={decorate}] (4.5,1) arc(-180:0: .5 and .25);
  \draw[black!40] (6.5,1) arc(180:0: .5 and .25)node[midway, above]{$c_{i+3}$};
  \draw[postaction={decorate}, black!40] (6.5,1) arc(-180:0: .5 and .25);
   
  \draw[black!40] (0,.5) to[out=0,in=-90] (.5,1);
  \draw[black!40] (1.5,1) to[out=-90,in=180] (2,.5);
  \draw[black!40] (0,-.5) to[out=0,in=180] (2,-.5);
  
  \draw (2,.5) to[out=0,in=-90] (2.5,1);
  \draw (3.5,1) to[out=-90,in=180] (4,.5);
  \draw (2,-.5) to[out=0,in=180] (4,-.5);
  
  \draw (4,.5) to[out=0,in=-90] (4.5,1);
  \draw (5.5,1) to[out=-90,in=180] (6,.5);
  \draw (4,-.5) to[out=0,in=180] (6,-.5);
  
  \draw[black!40] (6,.5) to[out=0,in=-90] (6.5,1);
  \draw[black!40] (7.5,1) to[out=-90,in=180] (8,.5);
  \draw[black!40] (6,-.5) to[out=0,in=180] (8,-.5);
\end{tikzpicture}
\vspace{2mm}
\end{center}
The sphere $\surface$ admits $n-3$ such sub-spheres, one for each pants curve in $\mathcal{B}$. We follow the convention of writing $b_0=c_1^{-1}$ and $\beta_0=2\pi-\alpha_1$, as well as  $b_{n-2}=c_n$ and $\beta_{n-2}=\alpha_n$. The conjugacy class of the restriction of $\rho$ to $\surface^{(i)}$ belongs to the relative character variety
\[
\Rep_{\alpha^{(i)}}(\surface^{(i)},\psl),
\]
where $\alpha^{(i)}=(2\pi-\beta_{i-1},\alpha_{i+1},\alpha_{i+2},\beta_{i+1})$. The restriction of $\rho$ to $\surface^{(i)}$ is obviously still totally elliptic. 

Let $\mathcal B_i$ denote the pants decomposition of $\surface^{(i)}$ consisting of the sole pants curve $b_i$.
The $\mathcal B_i$-triangle chain of the restriction of $\rho$ to $\surface^{(i)}$ is made of the two triangles that share the vertex $B_i$ in the $\mathcal{B}$-triangle chain of $\rho$. These two triangles are non-degenerate by assumption. In particular, this means that none of the restrictions of $\rho$ to any of the spheres $\surface^{(i)}$ is orthogonal. So, by the discussion of Section~\ref{sec:n=4}, all the restrictions are DT representations. This means that the two triangles in the $\mathcal B_i$-triangle chain of $\rho$ have the same orientation by Proposition~\ref{prop:DT-if-triangle-coherently-oriented}. Now, because the $\mathcal{B}$-triangle chain of $\rho$ is made of the $\mathcal B_i$-triangle chain of each restriction, with an overlap of one non-degenerate triangle for consecutive indices, we conclude that all the triangles in the $\mathcal B$-triangle chain of $\rho$ have the same orientation. 
\end{proof}

Proposition~\ref{prop:totally-elliptic-implies-DT} finishes the proof of Theorem~\ref{thm-intro}.

\section{Totally elliptic representations in $\PSL_2\C$}\label{sec:PSL(2,C)}

\subsection{Overview}
The goal of this section is to classify totally elliptic surface group representations into $\PSL_2\C$ and prove Theorems~\ref{thm-intro-complex} and~\ref{thm-intro-complex-reducible}. We will first prove that irreducible totally elliptic representations into $\PSL_2\C$ are either unitary or DT representations (Section~\ref{sec:irreducible}), and then consider the case of reducible totally elliptic representations (Section~\ref{sec:reducible}).

\subsection{Irreducible representations} \label{sec:irreducible}
Recall that irreducible representations are those whose image is not conjugate to a subgroup of $\PSL_2\C$ made of upper triangular elements. 
\begin{prop}\label{prop-irreducible}
Let $\rho\colon\pi_1\surface\to\mathrm{PSL}_2\C$ be a reduced totally elliptic surface group representation. If $\rho$ is irreducible, then $\rho$ is either unitary, or it is conjugate to the composition of a DT representation with the inclusion $\psl\subset \mathrm{PSL}_2\C$.
\end{prop}
\begin{proof}
For every $\gamma\in\pi_1\surface$, choose a representative $A_\gamma\in\mathrm{SL}_2\C$ of $\rho(\gamma)$ and write $t_\gamma$ for the trace of $A_\gamma$. Since $\rho$ is totally elliptic, for every $\sigma\in\pi_1\surface$ that represents a simple closed curve on $\surface$, the number $t_\sigma$ is real. Using trace relations for $2\times 2$ matrices, it is actually possible to write every trace $t_\gamma$ as a polynomial expression in the traces $t_\sigma$, where $\sigma$ ranges over fundamental group elements representing simple closed curves. This was initially observed by Goldman--Xia~\cite[Section~2]{goldman-xia} as a consequence of classical results by Procesi~\cite{procesi}. In particular, this means that the traces $t_\gamma$ are real numbers for all $\gamma\in\pi_1\surface$. Since $\rho$ is irreducible, it implies that the image of $\rho$ is conjugate into a real form of $\PSL_2\C$ (see for instance~\cite[Corollary~3.2.5]{reid}). The real forms of $\PSL_2\C$ being $\PU(2)$ and $\psl$, we conclude that $\rho$ is either unitary or conjugate to a representation into $\psl$. In the latter case, Theorem~\ref{thm-intro} applies and shows that $\rho$ is conjugate to the composition of a DT representation with the inclusion $\psl\subset \PSL_2\C$.
\end{proof}

\subsection{Reducible representations}\label{sec:reducible}
If $\rho\colon\pi_1\surface\to\mathrm{PSL}_2\C$ is reducible, then, up to conjugating $\rho$, we may assume that it is upper triangular. For every $\gamma\in\pi_1\surface$, $\rho(\gamma)$ acts on $\CP^1=\C\cup\{\infty\}$ by Möbius transformations and fixes $\infty$. It thus restricts to an affine transformation $z\mapsto a(\gamma)z+b(\gamma)$ of $\C$. Note that $\rho(\gamma)$ is elliptic if and only if $|a(\gamma)|=1$ and $a(\gamma)\neq 1$, or $a(\gamma)=1$ and $b(\gamma)=0$. When this is the case, $\rho(\gamma)$ corresponds to a Euclidean rotation of the plane. Now, the commutator of two rotations is either trivial (when they fix the same point), or a non-trivial translation. This statement is the Euclidean analogue to Lemma~\ref{lem:commutator-hyperbolic}. Adapting the proof of Proposition~\ref{prop:totally-elliptic-implies-orthogonal}, we obtain the following.
 
\begin{prop}\label{prop-reducible}
Assume that $\surface$ is a surface of genus $g\geq 1$ with an arbitrary number of punctures. If $\rho\colon\pi_1\surface\to\mathrm{PSL}_2\C$ is a reducible totally elliptic representation, then the image of $\rho$ is conjugate to a subgroup of $\PU(2)$ made of diagonal elements.
\end{prop}

Propositions~\ref{prop-irreducible} and~\ref{prop-reducible} together prove Theorem~\ref{thm-intro-complex}. It remains to prove Theorem~\ref{thm-intro-complex-reducible} by exhibiting reduced totally elliptic representations $\pi_1\surface\to\mathrm{PSL}_2\C$ that are reducible but not unitary when $\surface$ is a sphere with $n\geq 3$ punctures.

\begin{ex}\label{example-reducible}
Let $c_1,\ldots,c_n$ be geometric generators of $\pi_1\surface$ satisfying $c_1\cdots c_n=1$, as on Figure~\ref{fig:geometric-generators}. Pick an injective morphism $a$ from the abelianization of $\pi_1\surface$ (which is isomorphic to $\Z^{n-1}$) to the unit circle in $\C$. For $i=2,\ldots,n$, choose arbitrary numbers $b(c_i)\in \C$. There is a unique number $b(c_1)\in\C$ such that the map $\rho\colon\pi_1\Sigma\to\mathrm{PSL}_2\C$ defined by letting $\rho(c_i)$ be the affine transformation $z\mapsto a(c_i)z+b(c_i)$ is a representation. If we chose $b(c_2),\ldots,b(c_{n})$ such that $\rho(c_2),\ldots,\rho(c_{n})$, seen as Euclidean rotations of the plane, do not all fix the same point, then $\rho$ has non-abelian image, hence is not unitary.

We further claim that such a $\rho$ is a reduced totally elliptic representation. If $\gamma\in\pi_1\surface$ represents a simple closed curve, then either $\gamma$ or $\gamma^{-1}$ is equal to a product of the form $c_{i_1}\cdots c_{i_k}$ in the abelianization of $\pi_1\surface$ for some $1\leq k\leq n-2$ and $1\leq i_1 < \cdots < i_k\leq n$. In particular, $\gamma$ is non-trivial in the abelianization and thus $a(\gamma)\neq 1$, proving that $\rho(\gamma)$ is regular elliptic.
\end{ex}

\bibliographystyle{amsalpha}
\bibliography{references}

\end{document}